\documentclass[11pt]{article}
\usepackage[T1]{fontenc}
\usepackage{lmodern}
\usepackage{amssymb,amsmath,amsthm,dsfont}
\usepackage{bbm}
\usepackage{hyperref}
\usepackage{graphicx}
\usepackage{color}
\usepackage[a4paper,margin=2.5cm]{geometry}
\usepackage{authblk}

\newtheorem{thm}{Theorem}[section]
\newtheorem{corr}[thm]{Corollary}
\newtheorem{lem}[thm]{Lemma}
\newtheorem{prop}[thm]{Proposition}
\newtheorem{rmk}{Remark}

\def\conv{\ast}

\def\R{\mathbb{R}}
\def\P{\mathbb{P}}

\def\E{\mathbb{E}}
\def\N{\mathbb{N}}
\def\diffd{\mathrm{d}}
\newcommand{\indic}[1]{\mathds{1}\raisebox{-.4ex}{$\scriptstyle\{#1\}$}}
\newcommand{\uppar}[1]{^{\scriptscriptstyle (#1)}}

\numberwithin{equation}{section}

\begin{document}

\title{Global existence for a free boundary problem \\ of  Fisher-KPP type.}

\author{Julien Berestycki\thanks{\texttt{julien.berestycki@stats.ox.ac.uk}, Department of Statistics, University of Oxford.},
\'Eric Brunet\thanks{\texttt{Eric.Brunet@lps.ens.fr}, Laboratoire de Physique Statistique, Ecole Normale Sup\'erieure, PSL Research
University; Universit\'e Paris Diderot Sorbonne Paris-Cit\'e; Sorbonne Universit\'es UPMC Univ Paris 06; CNRS.},
Sarah Penington\thanks{\texttt{penington@maths.ox.ac.uk}, Mathematical Institute, University of Oxford.}}

\date{\today}
\maketitle

\begin{abstract}

Motivated by the study of branching particle systems with selection, we establish global existence for the solution $(u,\mu)$ of the free boundary problem 
\[
\begin{cases}
\partial_t u =\partial^2_{x} u +u & \text{for $t>0$ and $x>\mu_t$,}\\
u(x,t)=1 &\text{for $t>0$ and $x \leq \mu_t$}, \\
\partial_x u(\mu_t,t)=0 & \text{for $t>0$}, \\
u(x,0)=v(x) &\text{for $x\in \R$},
\end{cases}
\]
when the initial condition $v:\R\to[0,1]$ is non-increasing with $v(x) \to 0$ as $x\to
\infty$ and $v(x)\to 1$ as $x\to -\infty$.  We construct the solution as the limit of a sequence $(u_n)_{n\ge
1}$, where each $ u_n$ is the solution of a Fisher-KPP equation with same
initial condition, but with a different non-linear term. 

Recent results of  De Masi \textit{et al.}~\cite{DeMasi2017a} show that this global solution can be identified with the hydrodynamic limit of the so-called $N$-BBM, {\it i.e.} a branching Brownian motion in which the population size is kept constant equal to $N$ by killing the leftmost particle at each branching event.

\end{abstract}

\section{Main results and introduction}

We establish global existence for a  free boundary problem of Fisher-KPP type:
\begin{thm}\label{main thm}
Let $v:\R\to[0,1]$ be a non-increasing function such that $v(x)\to 0$ as $x\to
\infty$ and $v(x)\to 1$ as $x\to-\infty$. 
Let
$\mu_0=\inf\{x\in \R:v(x)<1\}\in \{-\infty\}\cup \R$.
Then there exists
a unique classical solution $(u,\mu)$ with $u\in[0,1]$
to the following free boundary problem:
\begin{equation}\label{FBP}\tag{FBP}
\begin{cases}
\partial_t u =\partial^2_{x} u +u & \text{for $t>0$ and $x>\mu_t$,}\\
u(x,t)=1 &\text{for $t>0$ and $x \leq \mu_t$}, \\
\partial_x u(\mu_t,t)=0 & \text{for $t>0$}, \\
u(x,0)=v(x) &\text{for $x\in \R$}.
\end{cases}
\end{equation}
Furthermore, this unique solution satisfies the following properties:
\begin{itemize}
\item
For every $t>0$, $u(\cdot, t)\in C^1(\R)$, and $\partial_x u \in C(\R\times (0,\infty))$.
\item If $v\uppar1 \le v\uppar2$ are two valid initial conditions and
$(u\uppar i,\mu\uppar i)$ is the solution with initial condition $v\uppar i$, then
$u\uppar 1\le u\uppar 2$ and $\mu\uppar 1\le\mu\uppar 2$. 
\end{itemize}
\end{thm}
We say that $(u,\mu)$ is a {\it classical solution} to
\eqref{FBP} above if $\mu_t\in \R$ $\forall t>0$, $t\mapsto \mu_t$ is
continuous,  $u:\R \times (0,\infty)\to [0,1]$, $u\in C^{2,1}(\{(x,t):t>0,\, x>\mu_t\})\cap C(\R\times (0,\infty))$, $(u,\mu)$ satisfies  the equation
\eqref{FBP}, and $u(\cdot,t)\to v(\cdot)$ in $L^1_\text{loc}$ as $t\searrow 0$. 
\begin{rmk}
We will show that $u(t,x) \to v(x)$ at all points of continuity of $v$ as $t\searrow 0$ (since $v$ is non-increasing, it is differentiable almost everywhere).
\end{rmk}
\begin{rmk}
If instead $v(x)\to l>0$ as $x\to\infty$, then a classical
solution $(u,\mu)$ of~\eqref{FBP} exists for $t<t_c=-\log l$, with $\mu_t\to\infty$ as
$t\nearrow t_c$.
\end{rmk}
\begin{rmk}
As discussed below, the condition that $v$ is
non-increasing can be relaxed to some extent.
\end{rmk}

Our motivation for studying the problem \eqref{FBP} stems from its connection with the so-called $N$-BBM, a variant of branching Brownian motion in $\R$ in which the number of active particles is kept constant (and equal to $N$) by removing the leftmost particle each time a particle branches. More details are given in Section~\ref{subsec:context} below, but in a nutshell, De Masi {\it et al.}~\cite{DeMasi2017a} show that as $N \to \infty$, under appropriate conditions on the initial configuration of particles, the $N$-BBM has a hydrodynamic limit whose cumulative distribution can be identified with the solution of~\eqref{FBP}, provided such a solution exists.

\medskip

The overall idea behind the proof is to construct $u$ as the limit of
a sequence of functions $u_n$, where, for each $n$,
$u_n$ satisfies an $n$-dependent
non-linear equation, but where all the $u_n$ have the same initial
condition.
More precisely, let $v:\R\to[0,1]$ be a measurable
function  and, for $n\ge 2$,
let  $(u_n(x,t), x\in \R, t\ge 0)$ be the
solution to 
\begin{equation}
\begin{cases}
\partial_t u_n=\partial_x^2u_n + u_n -u_n^n&\text{for $x\in \R$ and
$t>0$},\\
u_n(x,0)=v(x) &\text{for $x\in \R$} .
\end{cases}
\label{mainn}
\end{equation}
For each $n\ge 2,$ this is a version of the celebrated 
Fisher-KPP equation about which much is known (see e.g. \cite{kpp,Aronson1975,McKean1975,Uchiyama1978,Hamel2013,Nolen2016}). In particular,
\begin{itemize}
\item $u_n$ exists and is unique, 
\item $u_n(x,t)\in (0,1)$ for $x\in \R$ and $t> 0$  (unless $v \equiv 0$ or $v \equiv 1$).
\end{itemize}
Since the comparison principle applies, we see furthermore that for every
$x\in \R, t> 0$ fixed, the sequence $n\mapsto u_n(x,t)$ is increasing. 
Therefore, the following pointwise limit is well defined:
\begin{equation}
u(x,t):=\lim_{n\to\infty} u_n(x,t),
\label{mainu}
\end{equation}
with $u(x,t)\in(0,1]$ for $t>0$ (unless $v \equiv 0$). Indeed, in most of the cases we
are interested in, there are regions where $u(x,t)=1$.

We have the following results on $u$:
\begin{thm}\label{thm u}
Let $v:\R\to[0,1]$ be a measurable function. The function $u(x,t)$ as defined by~\eqref{mainn} and~\eqref{mainu} satisfies the following properties:
\begin{itemize}
\item $u$ is continuous on $\R\times (0,\infty)$ and, for $t>0$, $u(\cdot,t)$ is Lipschitz continuous.
\item $u(\cdot,t)\to v(\cdot)$ in $L^1_\text{loc}$ as $t\searrow0$, and if $v$ is continuous at $x$ then $u(t,x)\to v(x)$ as $t\searrow 0$.
\item At  any $(x,t)$ with $t>0$ such that $u(x,t)<1$, the function $u$ is continuously differentiable
in $t$ and twice continuously
differentiable in $x$, and
satisfies
$$\partial_t u = \partial_x^2u+u.$$
\item $u$ satisfies the following semigroup property: for any $t> 0$
and any $t_0\ge0$, $u(\cdot,t+t_0)$ can be obtained as the solution at time~$t$ 
to
\eqref{mainn} and \eqref{mainu} with an initial condition
$u(\cdot,t_0)$.
\item If $v^{(1)} \le v^{(2)}$ are two measurable functions and
$u^{(i)}$ is the solution to~\eqref{mainn} and~\eqref{mainu} with initial condition $v^{(i)}$, then
$u^{(1)}\le u^{(2)}$. 
\end{itemize}
\end{thm}

The existence result in Theorem~\ref{main thm} is then a consequence of the
following result:
\begin{prop}  \label{prop mu} 
Suppose that $v$ (and $\mu_0$)
is as in Theorem \ref{main thm}, and define $u(x,t)$ as in~\eqref{mainn} and~\eqref{mainu}. Then
there exists a map $t\mapsto \mu_t$ with $\mu_t\in \R$ $\forall t>0$ and $\mu_t\to\mu_0$ as $t\searrow 0$
such that
\begin{equation}\label{uandmu}
 u(x,t)=1 \, \Leftrightarrow \, x\le \mu_t\qquad\text{for $t>0$}.
\end{equation}
Furthermore, $t\mapsto\mu_t$ is continuous and
$u(\cdot,t)\in C^1(\R)$ for $t>0$ with $\partial_x u \in C(\R\times (0,\infty))$.
\end{prop}
By combining Theorem~\ref{thm u} and Proposition~\ref{prop mu}, we have that if $v$ is as in Theorem~\ref{main thm} then $(u,\mu)$ is a classical solution of \eqref{FBP}.

\begin{rmk}
For an arbitrary measurable initial condition $v$, for $t>0$, $u(\cdot,t)$ is obviously $C^1$ in the interior of the
region where $u=1$, and by Theorem~\ref{thm u} it is $C^1$ in the region where $u<1$. The difficulty in proving Proposition~\ref{prop mu} is to
show that $u(\cdot,t)$ is also $C^1$ at the boundary between these two domains.
\end{rmk}
\begin{rmk}
It turns out that the proof that $u(\cdot,t)$ is $C^1$ holds whenever the topological boundary
between the (two-dimensional) domains $\{u=1\}$ and $\{u<1\}$ has measure zero. (In the case where $v$ is non-increasing, this 
is implied by the existence of a continuous map $t\mapsto\mu_t$ satisfying \eqref{uandmu}.)
This means that it should be possible to show that $u(\cdot,t)$ is $C^1$ for any ``reasonable''
initial condition.
\end{rmk}
\begin{rmk}
The condition that $v$ is non-increasing in Theorem~\ref{main thm} is only
used in the proof of Proposition~\ref{prop mu} to show the
existence of a continuous boundary $t\mapsto\mu_t$ as in~ \eqref{uandmu}.
\end{rmk}

The idea of using the limit of $(u_n)_{n \ge 1}$ as the solution to \eqref{FBP}
first appeared in \cite{Berestycki2017} and the present article puts this
intuition on a rigorous footing.

The rest of the article is organised as follows: the next section is devoted
to putting  our result in the context of several recent works on related
problems. Next, in Section~\ref{sec:FK}, we present the precise
versions of the Feynman-Kac representation that we shall use in the rest of the proof. 
The proof of one of these Feynman-Kac results will be postponed until Section~\ref{sec:FK proofs}.
We establish
Theorem~\ref{thm u} in Section~\ref{sec:thm u}, and in
Section~\ref{sec:proof1} we prove
Proposition~\ref{prop mu}.
In Section~\ref{sec:uniq}, we complete the proof of Theorem~\ref{main thm} by proving the uniqueness of the classical
solution of~\eqref{FBP}.


\section{Context} \label{subsec:context}

Let $\omega$ be a probability measure on $\R$. 
Then define $v:\R \to [0,1]$ by setting
$$ v(x)=\omega [x,\infty).$$
Note that $v$ is non-increasing, and that $v(x)\to 0$ as $x\to \infty$ and $v(x)\to 1$ as $x\to -\infty$.
Therefore, by Theorem~\ref{main thm}, there exists a unique classical solution $(u,\mu)$ to the free boundary problem~\eqref{FBP}, and $\partial_x u$ is continuous on $\R \times (0,\infty)$.

Let $\rho =-\partial_x u$.
The following result is an easy consequence of Theorem~\ref{main thm} and its proof.
\begin{corr} \label{cor:rhofbp}
Let $\omega$ be a probability measure on $\R$
and let $\mu_0=\inf\{x\in \R:\omega [x,\infty)<1\}\in \R \cup \{-\infty\}$.
Then $(\rho, \mu)$ constructed as above from the solution of~\eqref{FBP} with initial condition $v(x)=\omega[x,\infty)$ is the unique classical
solution with $\rho \geq 0$ to the following free boundary problem:
\begin{equation}\label{FBP'}\tag{FBP$'$}
\begin{cases}
\partial_t \rho =\partial^2_{x} \rho +\rho & \text{for $t>0$ and $x>\mu_t$,}\\
\rho(\mu_t,t)=0, \quad \int_{\mu_t}^\infty \rho(y,t)\,\diffd y=1 & \text{for $t>0$}, \\
\rho(\cdot, t)\diffd\lambda \to \diffd\omega(\cdot) &\text{in the vague topology as $t\searrow 0$}.
\end{cases}
\end{equation}
\end{corr}
We say that $(\rho,\mu)$ is a {\it classical solution} to~\eqref{FBP'} above if $\mu_t \in \R$ $\forall t>0$, $t\mapsto \mu_t$ is continuous,  $\rho:\R \times (0,\infty) \to [0,\infty)$, $\rho\in C^{2,1}(\{(x,t):t>0,\, x>\mu_t\})\cap C(\R\times (0,\infty))$, and $(\rho,\mu)$ satisfies  the equation
\eqref{FBP'}. 

This result improves on a recent result of Lee~\cite{Lee2017}, where local existence of a solution to~\eqref{FBP'} is shown (i.e.~existence of a solution on a time interval $[0,T]$ for some $T>0$), under the additional assumptions that $\omega$ is absolutely continuous with respect to Lebesgue measure with probability density $\phi\in C^2_c(\R)$, and that there exists $\mu_0\in \R$ such that $\phi(\mu_0)=0$, $\phi '(\mu_0)=1$ and $\int_{\mu_0}^\infty \phi (x)\,\diffd x=1$.

\medskip 

In \cite{DeMasi2017a}, De Masi et al. study the hydrodynamic limit of the $N$-BBM and its relationship with the free boundary problem~\eqref{FBP'}. The $N$-BBM is a variant of branching Brownian motion in which the number of active particles is kept constant (and equal to $N$) by removing the leftmost particle each time a particle branches. 

We shall now define this particle system more precisely.
Suppose that $\phi \in L^1(\R)$ is a probability density function which satisfies (a) $\|\phi\|_\infty<\infty$ and (b) $\int_r^\infty \phi(x)\,\diffd x=1$ for some $r\in \R$.
Let $X^1_0,\ldots, X^N_0$ be i.i.d.~with density $\phi$. 
At time $0$, the $N$-BBM consists of $N$ particles at locations $X^1_0,\ldots, X^N_0$.
These particles move independently according to Brownian motions, and each particle independently at rate $1$ creates a new particle at its current location.
Whenever a new particle is created, the leftmost particle is removed from the particle system.

Let $X_t=\{X^1_t,\ldots, X^N_t\}$ denote the set of particle locations at time $t$.
Let $\pi_t^{(N)}$ be the empirical distribution induced by the particle system at time $t$, i.e.~for $A\subset \R$, let
\[
\pi_t^{(N)}(A)= \frac1N |X_t \cap A|.
\]
De Masi et al.~prove in~\cite{DeMasi2017a} that for each $t\geq 0$ there exists a probability density function $\psi(\cdot,t):\R\to [0,\infty)$ such that, for any $a\in \R$,
$$
\lim_{N\to \infty} \pi^N_t [a,\infty)
=\int_a^\infty \psi (r,t) \,\diffd r 
\qquad \text{ a.s. and in }L^1.
$$
Moreover, they show that if $(\rho,\mu)$ is a classical solution of~\eqref{FBP'} with initial condition $\omega$ given by $\diffd\omega = \phi\, \diffd\lambda$ then $\psi=\rho$.
The following result is then a direct consequence of Theorems~1 and~2 in~\cite{DeMasi2017a} and our Corollary~\ref{cor:rhofbp}.
\begin{corr}
Suppose $\phi \in L^1(\R)$ is a probability density function with $\|\phi\|_\infty<\infty$ and $\int_r^\infty \phi(x)\,\diffd x=1$ for some $r\in \R$.
Construct an $N$-BBM with initial particle locations given by i.i.d.~samples from $\phi$, as defined above.
Let $\pi_t^{(N)}$ denote the empirical distribution induced by the particle system at time $t$.
Then 
for any $t\geq 0$ and $a\in \R$,
$$
\lim_{N\to \infty} \pi^N_t [a,\infty)
=\int_a^\infty \rho (r,t) \,\diffd r =u(a,t)
\qquad \text{ a.s. and in }L^1,
$$
where $(u,\mu)$ is the solution of~\eqref{FBP} with initial condition $v$ given by $v(x)=\int_x^\infty \phi(y)\,\diffd y,$ and $\rho=-\partial _x u$. 
\end{corr}

\medskip

Lee \cite{Lee2017} points out that \eqref{FBP'} can be reformulated
as a variant of the Stefan problem;
let $(\rho, \mu)$ be a solution of \eqref{FBP'} and define $w(x,t):=
e^{-t}\partial_x \rho(x,t)$. Then under some regularity assumptions,
 $(w,\mu)$ solves
\begin{equation}\label{FBPstefan}  \tag{Stefan}
\begin{cases}
\partial_t w =\partial^2_{x} w  & \text{for $t>0$ and $x>\mu_t$,}\\
w(\mu_t,t)=e^{-t},\, \partial_t \mu_t = -\frac12 e^t \partial_x w(\mu_t,t)  & \text{for } t > 0.
\end{cases}
\end{equation}
The Stefan problem describes the phase change of a material and is one of the most
popular problems in the moving boundary problem literature. Typically, it requires
solving heat equations for the temperature in the two phases (e.g. solid and
liquid), while the position of the front separating them, the moving boundary, is
determined from an energy balance referred to as the Stefan condition. The Stefan
problem has been studied in great detail since Lamé and Clapeyron formulated it in
the 19th century \cite{Lame}. There are several reference books that the reader may consult such as the recent and up-to-date book \cite{Gupta2003}.

The free boundary problem \eqref{FBP} was studied at a heuristic level in \cite{Berestycki2017}. 
The main tool there was the following
relation, which can be proved rigorously:
\begin{lem}
Let $v$ be as in Theorem \ref{main thm} with $\mu_0 \in \R$ and such that
$\gamma:=\sup\big\{ r: \int_{\mu_0}^\infty v(x)e^{rx}\,\diffd
x <\infty\big\}\ge0.$ Let $(u,\mu)$ be the classical solution to \eqref{FBP}.
Then, for $r<\min(\gamma,1)$, and for any $t\ge0$,
\begin{equation}
1+ r \int_{0}^\infty\diffd x\, u(\mu_t+x,t)e^{rx}
= \int_0^\infty \diffd s\ 
e^{r(\mu_{t+s}-\mu_t)-(1+r^2)s}<\infty
.
\label{magic}
\end{equation}
\end{lem}
Although the proof is not very difficult, we omit it  from the present work as it is not our main focus here; 
the main ideas can be found in \cite{Berestycki2017, brunet2015exactly}.

For instance, take a step initial condition $v(x)=\indic{x \le 0}$. Using
\eqref{magic} with $t=0$ and $r=1-\epsilon$ gives
\begin{equation}
\int_0^\infty\diffd s\, e^{ -\epsilon^2 s +(1-\epsilon)(\mu_s -2s)}=1\qquad\forall
\epsilon>0.
\label{v=0}
\end{equation}
Whether \eqref{v=0} completely characterizes  the function $\mu$ or not is an
open question.


\section{Feynman-Kac formulae} \label{sec:FK}

In this section, we state  the versions of the
Feynman-Kac formula which we shall use repeatedly in the rest of the paper.
So as not to interrupt the flow of the main argument, the proof for 
Proposition~\ref{lem:FKw} is postponed to Section \ref{sec:FK proofs}.

We introduce the heat kernel
\begin{equation} \label{ptdef}
p_t(x)=\frac1{\sqrt{4\pi t}}e^{-\frac{x^2}{4t}}.
\end{equation}
For $x\in \R$, we let $\P_x$ denote the probability measure under which $(B_t)_{t\geq 0}$ is a Brownian motion with diffusivity constant $\sqrt 2$ started at $x$. We let $\E_x$ denote the corresponding expectation.
The symbol $\conv$ denotes convolution; for instance,
$$p_t\conv v(x)=\int_{-\infty}^\infty \diffd y\,p_t(x-y)v(y)=\E_x[v(B_t)]$$
is the solution at time $t$ to the heat equation on $\R$ with an initial
condition $v$.

\begin{prop} \label{lem:FKw}
Suppose that $A\subseteq \R\times (0,\infty)$ is an open set,
and that $w:A\to \R$ is $C^{2,1}$ and bounded, and satisfies
\begin{equation} \label{equ w}
\partial_t w = \partial^2_x w +Kw +S
\quad \text{ for }(x,t)\in A,
\end{equation}
where $K:A\to \R$, 
$S:A \to \R$
are continuous and bounded.
Then, if one of the conditions below is met, we have the following
representation for $w(x,t)$ with $(x,t)\in A$:
\begin{equation} \label{eq:FKgen}
w(x,t)=\E_x \left[w(B_{\tau},t-\tau)e^{\int_0^{\tau} K(B_s,t-s)\,\diffd s} 
  + \int_0^\tau\diffd r\, S(B_r,t-r)e^{\int_0^{r   } K(B_s,t-s)\,\diffd s}\right],
\end{equation}
where $\tau$ is a stopping time for $(B_s)_{s\geq 0}$.

For the representation \eqref{eq:FKgen} to hold, it is sufficient to have
one the following:
\begin{enumerate}
\item The stopping time $\tau$ is such that $(B_s,t-s) \in A$ for all $s
\le\tau$,
\item 
The set $A$ is given by $A=\big\{(x,t):t\in(0,T)\text{ and
}x>\mu_t\big\}$ for some $T>0$ and some continuous boundary $t\mapsto\mu_t$
with $\mu_t\in\R\cup\{-\infty\}$ $\forall t\in [0,T]$,
the stopping time $\tau$ is given by $\tau=\inf\big\{s \ge0: B_s\le\mu_{t-s}\big\}\wedge t$
(the first time at which $(B_\tau,t-\tau)\in\partial A$) and,
furthermore,
$w$ is defined and bounded on $\bar A$, continuous on $\bar A \cap (\R\times (0,\infty))$ and satisfies $w(\cdot,t)\to w(\cdot,0)$ in
$L^1_\text{loc}$ as $t\searrow0$.
\end{enumerate}
\end{prop}
Although this is a very classical result, we give a proof in Section~\ref{sec:FK proofs} for the sake
of completeness and because we could not find a exact statement with
stopping times or a discontinuous initial condition in the literature. 
The proof that~\eqref{eq:FKgen} holds under condition~1  essentially follows the proof of Theorem~4.3.2 in~\cite{Durrett1996}.



Proposition~\ref{lem:FKw} gives some useful representations for the $u_n$ defined
in \eqref{mainn}.
\begin{corr} Let $v:\R\to[0,1]$ be measurable, let $n\ge2$ and let $u_n(x,t)$ denote the solution
to~\eqref{mainn}. Then by Proposition~\ref{lem:FKw}:
\begin{itemize}
\item taking $K=1-u_n^{n-1}$ and $S=0$, for $\tau$ a stopping time with
$\tau< t$:
\begin{equation} \label{eq:FKtau}
u_n(x,t)=\E_x \left[u_n(B_{\tau},t-\tau)e^{\int_0^{\tau}
(1-u_n^{n-1}(B_s,t-s))\diffd s} \right].
\end{equation}

\item taking $K=1-u_n^{n-1}$, $S=0$, and $\tau=t$:
\begin{equation} \label{eq:FKt}
u_n(x,t)=\E_x \left[v(B_t)e^{\int_0^{t}
(1-u_n^{n-1}(B_s,t-s))\diffd s} \right].
\end{equation}

\item taking $K=0$, $S=u_n-u_n^n$ and $\tau=t$:
\begin{equation}
\label{unrep}
\begin{aligned}u_n(x,t)&=
\E_x \left[v(B_{t}) 
  + \int_0^t\diffd r\, \big[u_n(B_r,t-r)-u_n^n(B_r,t-r)\big]\right]
\\
&=p_t\conv v(x) + \int_0^t\diffd r\, p_r\conv \big[u_n(x,t-r)-u_n^n
(x,t-r)\big]
\end{aligned} \end{equation}
\item taking $K=1$, $S=-u_n^n$ and $\tau=t$:
\begin{equation}
\label{unrep2}
\begin{aligned}u_n(x,t)&=
\E_x \left[v(B_{t})e^t 
  - \int_0^t\diffd r\, e^ru_n^n(B_r,t-r)\right]
\\
&=e^t p_t\conv v(x) - \int_0^t\diffd r\, e^rp_r\conv u_n^n
(x,t-r).
\end{aligned} \end{equation}
\end{itemize}
\end{corr}
\begin{proof}
This is a direct consequence of the previous result.
\end{proof}

We will also use the following representation for solutions of the free boundary problem~\eqref{FBP}:
\begin{corr}
If $v$ is as in Theorem~\ref{main thm} and $(u,\mu)$ is a classical solution
of~\eqref{FBP} with initial condition $v$, then for $t> 0$ and $x\in \R$,
\begin{equation} \label{eq:uniq(star)}
u(x,t)=\E_x \left[e^\tau \indic{\tau<t}+e^t v(B_t)\indic{\tau = t}\right],
\end{equation}
where
$\tau=\inf\{s\geq 0:B_s\leq \mu_{t-s}\} \wedge t$.
\label{lem:FKfbp}
\end{corr}
\begin{proof}
This is a direct application of Proposition~\ref{lem:FKw} under condition~2,
with $K=1$ and $S=0$.
\end{proof}

Finally, we use the following result to recognise solutions to partial differential
equations:
\begin{lem} \label{lem:FKreverse}
Suppose that $a<b$, $t_0<t_1$, and that $g:[a,b]\times [t_0,t_1]\to [0,\infty)$ is continuous and for $x\in [a,b]$ and $t\in [t_0,t_1]$,
$$
g(x,t)=\E_x \left[g(B_\tau, t-\tau) e^\tau \right],
$$ 
where $\tau = \inf\{s\geq 0:B_s\in \{a,b\}\} \wedge (t-t_0)$.
Then $g \in C^{2,1}((a,b)\times (t_0,t_1))$ with
$$
\partial_t g=\partial^2_x g+g \qquad \text{ for } (x,t)\in (a,b)\times (t_0,t_1).
$$
\end{lem}
\begin{proof}
The proof is the same as the proof of Exercise 4.3.15 in~\cite{karatzasshreve}, where an outline proof is given.
\end{proof}


\section{Proof of Theorem~\ref{thm u}}\label{sec:thm u}

In this section, we suppose $v:\R\to [0,1]$ is measurable. Let $u_n$ denote the solution of~\eqref{mainn} and define $u$ as in~\eqref{mainu}.
We shall use the following basic results on the smoothing effect of convolution with the 
heat kernel $p_t$ as introduced in~\eqref{ptdef}.
\begin{lem}
Suppose $t>0$.
\begin{enumerate}
\item
If $x\mapsto a(x)$ is bounded, then $x\mapsto p_t\conv a(x)$ is $C^\infty$ and $(p_t\conv a)^{(n)}(x)=p_t^{(n)}\conv a(x)$.
\item
If $(x,s)\mapsto b(x,s)$ is such that $b_s :=\|b(\cdot,s)\|_\infty<\infty$ for each $s\in (0,t)$, and the map $s\mapsto \frac{b_s}{\sqrt{s}}$ is integrable on $[0,t]$, then $f(x):=\int_0^t\diffd s\,\int_{-\infty}^\infty \diffd
y\, p_s(x-y) b(y,s)=\int_0^t\diffd s\, p_s\conv b(x,s)$ is $C^1$ and 
$f'(x)=\int_0^t\diffd s\, p'_s\conv b(x,s)$.
\end{enumerate}
\label{lemsmooth}
\end{lem}
\begin{proof}
The first statement holds since for every $n\in\N$ and $t>0$,
there exists a polynomial function $q_{n,t}:\R\to \R$ such that
$
|p_t^{(n)}(x-y)|\leq |q_{n,t}(x-y)| e^{-(x-y)^2/(4t)}
$
$\forall x,y\in \R$.
Then for the second statement, we have that
$f_s(x):=p_s\ast b(x,s)$ is smooth, with
$$\big|f'_s(x)\big|
=\big|  p_s' \conv b(x,s)\big|
\le   \int_{-\infty}^\infty \diffd y\,  \big| p_s'(x-y) \big| b_s= \frac{b_s}{\sqrt{\pi s}}.$$
Since $s\mapsto \frac{b_s}{\sqrt{\pi s}}$ is integrable on $[0,t]$, the result follows.
\end{proof} 
The following result of Uchiyama provides a useful bound on the spatial derivative of $u_n$.
\begin{lem}[\cite{Uchiyama1978}, Section~4]
\label{lem:uchi}
For $x\in \R$ and $t>0$,
 \begin{equation}
|\partial_x u_n(x,t)|\le\frac1{\sqrt{ \pi t}} +\frac{\sqrt 8}{\sqrt\pi}.
\label{unprimebound}
\end{equation}
\end{lem}

\begin{proof}We briefly recall Uchiyama's proof. Using Lemma~\ref{lemsmooth} to differentiate~\eqref{unrep} with respect to $x$, and bounding the result  (using
$v\in[0,1]$ and $u_n-u_n^n\in[0,1]$) yields:
$$|\partial_x u_n(x,t)|\le \int_{-\infty}^\infty  \big|p_t'(x-y) \big|\diffd y+\int_0^t\diffd s\,\int_{-\infty}^\infty \diffd
y\, \big|p'_{s}(x-y) \big|
=\frac1{\sqrt{\pi t}}+2\frac{\sqrt t}{\sqrt\pi}.
$$
This bound reaches its minimum $\sqrt{8/\pi}$ at  $t=1/2$. 
For $t\leq 1/2$, the result follows immediately.
For $t\ge1/2$,
since
$u_n(\cdot,t)$ is also the solution at time $1/2$ of~\eqref{mainn} with initial condition $u_n(\cdot,t-1/2)$, it follows that $|\partial_x u_n(x,t)|\le\sqrt{8/\pi}$ $\forall x\in \R$ and $t\geq 1/2$.
\end{proof}

In the following two lemmas, we prove the continuity of $u$.

\begin{lem}\label{lem:lipschitz}
For any $t>0$, the map $x\mapsto u(x,t)$ is
Lipschitz continuous, with Lipschitz constant 
$\frac1{\sqrt{ \pi t}} +\frac{\sqrt
8}{\sqrt\pi}$.
\end{lem}
\begin{proof}
For $x\in \R$ and $h>0$, we can write, using \eqref{unprimebound},
$$\Big|u_n(x+h,t)-u_n(x,t)\Big|\le \bigg(\frac1{\sqrt{ \pi t}} +\frac{\sqrt
8}{\sqrt\pi}\bigg)h.
$$
Then take the $n\to\infty$ limit to conclude.
\end{proof}

\begin{lem}\label{lem:cont t}
The map $(x,t)\mapsto u(x,t)$ is
continuous on $\R\times (0,\infty)$. Furthermore, $u(\cdot,t)\to v$ in $L^1_\text{loc}$ as $t\searrow 0$,
and if $v$ is continuous at $x$ then $u(x,t)\to v(x)$ as $t\searrow 0$.
\end{lem}
\begin{proof}
Using the bound $u_n-u_n^n\in[0,1]$ in the expression for $u_n$ in \eqref{unrep}, we have that for $x\in \R$, $t_0\geq 0$ and $t>0$,
$$u_n(x,t)-p_t\conv v(x) \in[0, t]\qquad\text{and}\qquad
  u_n(x,t_0+t) -p_t\conv u_n(x,t_0)\in[0, t], $$
where for the second
expression we used that $u_n(\cdot,t+t_0)$ is the solution at time $t$ of~\eqref{mainn} with initial condition $u_n(\cdot,t_0)$. Taking the $n\to\infty$ limit, it follows that
$$u(x,t)-p_t\conv v(x) \in[0,t]\qquad\text{and}\qquad
u(x,t_0+t) -p_t\conv u(x,t_0)\in[0,t]. $$
Since the solution to the heat equation
$p_t\conv v$ converges to $v$ in $L^1_\text{loc}$
as $t\searrow 0$,
we have that $u(\cdot,t)\to v$ in $L^1_\text{loc}$.
If $v$ is continuous at $x$, then $p_t \ast v(x) \to v(x)$ as $t\searrow 0$, and hence $u(x,t)\to v(x)$ as $t\searrow 0$.

It remains to prove that $u$ is continuous.
By Lemma~\ref{lem:lipschitz}, we have
$$\begin{aligned}
\left|p_t\conv u(x,t_0)-u(x,t_0) \right|
&=\left|\E_x\Big[u(B_t,t_0)-u(x,t_0)\Big]\right|
\\
&\leq \left( \frac1{\sqrt{ \pi t_0}} +\frac{\sqrt
8}{\sqrt\pi}\right)
\E_x \left[ |B_t -x|\right]
=\left( \frac1{\sqrt{ \pi t_0}} +\frac{\sqrt
8}{\sqrt\pi}\right)\sqrt{\frac{4t}{\pi}}.
\end{aligned}
$$
Therefore by the triangle inequality,
$$
\left| u(x,t_0+t)-u(x,t_0) \right|
\leq t +\left( \frac1{\sqrt{ \pi t_0}} +\frac{\sqrt
8}{\sqrt\pi}\right)\sqrt{\frac{4t}{\pi}}.
$$
Hence by the triangle inequality and then by Lemma~\ref{lem:lipschitz},
for $x_1$, $x_2\in \R$, $t_0\geq 0$ and $t>0$,
\begin{align*}
\left| u(x_1,t_0+t)-u(x_2,t_0) \right|
&\leq \left| u(x_1,t_0+t)-u(x_1,t_0) \right|
+ \left| u(x_1,t_0)-u(x_2,t_0) \right|
\\
&\leq t+\left( \frac1{\sqrt{ \pi t_0}} +\frac{\sqrt
8}{\sqrt\pi}\right)\left(\sqrt{\frac{4t}{\pi}}+|x_1-x_2|\right)
,
\end{align*}
and the result follows.
\end{proof}

We now turn to the semigroup property.
\begin{lem} \label{lem:utilde}
Suppose $v:\R\to[0,1]$, take $t_0\geq 0$ and, as throughout this section, let $u_n$ and
$u$ denote the functions defined in~\eqref{mainn} and~\eqref{mainu}. 
Furthermore, for $t\ge t_0$, let $u_{n;t_0}(\cdot,t)$ denote
the solution at time $t-t_0$ to \eqref{mainn} with the initial condition $v(\cdot)$
replaced by $u(\cdot,t_0)$. Then for $t\ge t_0$ and $x\in \R$,
$$\lim_{n\to\infty} u_{n;t_0}(x,t)=\lim_{n\to\infty} u_n(x,t)=u(x,t).$$
\end{lem}
\begin{proof}
Since $u_n(x,t_0)\leq u(x,t_0)$ $\forall x\in \R$, it follows by the
comparison principle that $u_n(x,t)\leq u_{n;t_0}(x,t)$ $\forall
x\in \R$, $t\geq t_0$. Then for $t\geq t_0$, by the
Feynman-Kac formula~\eqref{eq:FKt},
\begin{align*}
&u_{n;t_0}(x,t)- u_n(x,t)\\
& \quad =e^{t-t_0}\E_x\left[
u(B_{t-t_0},t_0)e^{-\int_0^{t-t_0} u_{n;t_0}^{n-1}(B_s,t-s)\,\diffd s}
-u_n(B_{t-t_0},t_0)e^{-\int_0^{t-t_0}u_n^{n-1}(B_s,t-s)\,\diffd s}
\right]\\
&\quad=e^{t-t_0}\E_x\left[
\Big(u(B_{t-t_0},t_0)-u_n(B_{t-t_0},t_0)\Big)e^{-\int_0^{t-t_0}
u_{n;t_0}^{n-1}(B_s,t-s)\,\diffd s}
\right]\\
&\quad\quad +e^{t-t_0}\E_x\left[
u_n(B_{t-t_0},t_0)\left(e^{-\int_0^{t-t_0}
u_{n;t_0}^{n-1}(B_s,t-s)\,\diffd s}-e^{-\int_0^{t-t_0}u_n^{n-1}(B_s,t-s)\,\diffd s}\right)
\right]\\
&\quad\leq e^{t-t_0}\E_x\Big[
u(B_{t-t_0},t_0)-u_n(B_{t-t_0},t_0)
\Big],
\end{align*}
where, in the last step, we used that $u_{n;t_0}\ge u_n$, $u_n \geq 0$, $u\ge u_n$ and
$u_{n;t_0}\ge0$. 
By dominated convergence, the right hand side converges to zero as $n\to\infty$,
and this completes the proof.
\end{proof}

At this point, it is convenient to introduce the two sets
\begin{equation} \label{eq:UandS}
\begin{aligned}
U&:=\big\{(x,t)\in\R\times(0,\infty): u(x,t)=1\big\}\\
\text{ and } \qquad S&:=\big\{(x,t)\in\R\times(0,\infty): u(x,t)<1\big\}.
\end{aligned}
\end{equation}
By the
continuity of $u$, the set $S$ is open.

The next proposition focuses on the set $S$, while
Proposition~\ref{prop:limunn} below is about the behaviour of $u_n$ in the
set~$U$.

\begin{prop} \label{prop:pde}
The map $u $ is $C^{2,1}$ on $S$ and satisfies
\begin{equation}
\partial_t u = \partial_x^2 u + u\qquad\text{on $S$.}
\end{equation}
\end{prop}
\begin{proof} 
Choose $(x,t)\in S$.
Let $a$, $b$, $t_0$ and $t_1$ be such that $x\in(a,b)$, $t\in(t_0,t_1)$ and
 $[a,b]\times[t_0,t_1]\subset S$.
By~\eqref{eq:FKtau}, we have that for $(x',t')\in [a,b]\times[t_0,t_1]$,
\begin{equation}
u_n(x',t')= \E_{x'}\Big[u_n(B_\tau,t'-\tau)e^{\int_0^\tau \big(1-
u_n^{n-1}(B_s,t'-s)\big)\diffd s}
\Big],
\label{unFK}
\end{equation}
where $\tau=(t'-t_0)\wedge \inf  \{s\geq 0: B_s \not \in (a,b) \}$ is
the time at which $(B_\tau,t'-\tau)$ hits the boundary of
$[a,b]\times[t_0,t_1]$.

We now take the $n\to\infty$ limit. For a given Brownian path $(B_s)_{s\geq 0}$, since $(B_s,t'-s)\in
S$ for $s\in[0,\tau]$, we have $u_n(B_s,t'-s)\to u(B_s,t'-s)<1$ as $n\to \infty$ for $s\in [0,\tau]$ and so, since $\tau \le t'-t_0$,
$$
\int_0^\tau 
u_n^{n-1}(B_s,t'-s)\diffd s \to 0 \quad \text{ as } n\to \infty.
$$
Hence by dominated convergence in \eqref{unFK},
\begin{equation} \label{eq:pdefk}
u(x',t')= \E_{x'}\Big[u(B_\tau,t'-\tau)e^\tau \Big]
\qquad\text{for any $(x',t')\in [a,b]\times [t_0,t_1]$}.
\end{equation}
The result then follows by Lemma \ref{lem:FKreverse}.
\end{proof}
To complete the proof of Theorem~\ref{mainu}, it only remains to
note that if $v\uppar1\le v\uppar2$ are two measurable functions, and
if $u_n\uppar i$ is the solution to \eqref{mainn} with initial condition
$v\uppar i$, then by the comparison principle $u_n\uppar1\le u_n\uppar2$ and hence
$u\uppar1\le u\uppar2$.

We finish this section by proving two more results on the behaviour of $u_n$ which will be used in the proof of Proposition~\ref{prop mu} in the next section, but which do not require any additional assumptions on $v$.
\begin{prop}\label{prop:limunn}
If $(x,t)$ is in the interior of $U$,
then
\begin{equation*}
\lim_{n\to\infty} u_n^n(x,t)=1.
\end{equation*}
\end{prop}
Before proving this result properly, we give a heuristic explanation.
As in the proof of Proposition~\ref{prop:pde}, choose a rectangle
$[a,b]\times[t_0,t_1]$ in the interior of $U$, and write \eqref{unFK} for
a point $(x',t')\in(a,b)\times(t_0,t_1)$. We take the limit
$n\to\infty$ again. By construction, $u_n(x',t')\to1$ and
$u_n(B_\tau,t'-\tau)\to1$, so we obtain
$$1=\lim_{n\to\infty}\E_{x'}
\Big[e^{\int_0^\tau \big(1-
u_n^{n-1}(B_s,t'-s)\big)\diffd s}
\Big].
$$
This equation strongly suggests the result, because if there were a region
where $\limsup_{n\to\infty} u_n^n<1$ which was visited by the paths $(B_s,t'-s)$ with
a strictly positive probability then the limiting expectation above would be larger
than~1. However, we were not able to turn this heuristic into a proper
proof of Proposition~\ref{prop:limunn}, so we used a completely different method.
\begin{proof}
Take $(x,t)$ in the interior of $U$.
For $\epsilon>0$, let
$$A=[-\epsilon^{0.49},\epsilon^{0.49}].$$
(The exponent $0.49$ could be any positive number smaller than $1/2$.)
Choose
$\epsilon$ sufficiently small that $[x-\epsilon^{0.49},x+\epsilon^{0.49}]\times[t-\epsilon,t]\subset U$.
Note that $u_n$ is a monotone sequence and converges pointwise to 1 on $[x-\epsilon^{0.49},x+\epsilon^{0.49}]\times[t-\epsilon,t]$.
Therefore, by Dini's theorem, we can
choose $n_0$ sufficiently large that
$u_n(x+y,t-\epsilon)>1-\frac\epsilon2$ for all $y\in A$ and all $n\ge n_0$.

Let $w_n(y,s)$ denote the solution to
\begin{equation} \label{w_n}
\begin{cases}\displaystyle\partial_s w_n=\partial_y^2
w_n+ w_n- w_n^n &\text{for $y\in \R$ and $s>0$, }\\[1ex]
\displaystyle
w_n(y,0)=\Big(1-\frac\epsilon2\Big)\indic{y\in A} &\text{for $y\in \R$.}\end{cases}
\end{equation}
Then, by the comparison principle, $u_n(x+y,t-\epsilon+s)\ge w_n(y,s)$ for
$n\ge n_0$, $s\ge0$ and $y\in \R$.

Heuristically, the domain $A$ is so ``large'' that, for times
$s\le \epsilon$, the solution $w_n$ behaves locally near $y=0$ as if started from a flat initial condition. 
This  suggests that $\partial_y^2 w_n(0,s)$ is very small for $s\in[0,\epsilon]$.
Indeed, starting from \eqref{unrep2} we have
$$w_n(y,s)=e^{s} p_s\conv w_n(y,0)
- \int_0^s\diffd r \, e^{s-r}  \int_{-\infty}^\infty \diffd z\, p_{s-r}(y-z) w_n^n(z,r).$$
Taking the derivative with respect to $y$, using Lemma~\ref{lemsmooth}, yields
$$
\partial_y w_n(y,s)=e^{s} p'_s\conv w_n(y,0)
- \int_0^s\diffd r \, e^{s-r}  \int_{-\infty}^\infty \diffd z\, p'_{s-r}(y-z) w_n^n(z,r).$$
Then integrating by parts with respect to $z$ in the second term, we have that
$$
\partial_y w_n(y,s)=e^{s} p'_s\conv w_n(y,0)
- \int_0^s\diffd r \, e^{s-r}  \int_{-\infty}^\infty \diffd z\, p_{s-r}(y-z) n \partial_z w_n(z,r) w_n^{n-1}(z,r).$$
Note that $|\partial_z w_n(z,r)|\leq \frac{1}{\sqrt{\pi r}}+\frac{\sqrt{8}}{\sqrt{\pi}}$ $\forall z\in \R$ by Lemma~\ref{lem:uchi}, and the map
$r\mapsto e^{s-r}\frac{1}{\sqrt{s-r}}\left(\frac{1}{\sqrt{\pi r}}+\frac{\sqrt{8}}{\sqrt{\pi}} \right)$
is integrable on $[0,s]$.
Hence by Lemma~\ref{lemsmooth}, we can take the derivative with respect to $y$ again, to obtain, at $y=0$,
$$\partial_y^2 w_n(0,s)
=e^{s} p_s''\conv w_n(0,0)
-n \int_0^s\diffd r\, e^{s-r}\int_{-\infty}^\infty \diffd z\, p'_{s-r}(-z) \partial_z  w
_n(z,r) w_n^{n-1}(z,r).$$
Clearly, $\partial_z w_n(z,r)$ has the opposite sign to
$z$, while $p'_{s-r}(-z)$ has the same sign as $z$.
Hence the double integral is negative and
$$\partial_y^2 w_n(0,s)
  \ge e^{s} p_s''\conv w_n(0,0) 
  =   \Big(1-\frac\epsilon2\Big) e^{s}2p_s'(\epsilon^{0.49})
  = - \Big(1-\frac\epsilon2\Big) e^{s}\frac{\epsilon^{0.49}}{2\sqrt\pi
s^{\frac32}}e^{-\frac{\epsilon^{0.98}}{4s}}.$$
The function $s\mapsto s^{-\frac32}
e^{-\epsilon^{0.98}/(4s)}$ reaches its maximum at
$s=\epsilon^{0.98}/6$ and is increasing on $[0,\epsilon^{0.98}/6)$.
 Thus, for $\epsilon$ small enough,  $s\mapsto s^{-\frac32}
e^{-\epsilon^{0.98}/(4s)}$ is increasing on $[0,\epsilon]$ and so
$$\partial_y^2 w_n(0,s)\ge 
- \epsilon^{-1.01}e^{-\frac{\epsilon^{-0.02}}4}\quad\text{for
$s\in[0,\epsilon]$.}$$
This bound is uniform in $n$ and 
goes to zero faster than $\epsilon$. Thus, we can choose $\epsilon$ small
enough that $\partial_y^2 w_n(0,s)>-\epsilon/2$ $\forall s\in[0,\epsilon]$.
We use this in {\eqref{w_n}} and obtain, by the comparison
principle, $w_n(0,s) \ge y_n(s)$ for $s\in[0,\epsilon]$, where $y_n$ is the solution of 
$$\partial_s y_n(s) = -\frac\epsilon2 +y_n(s)-y_n(s)^n,\qquad
y_n(0)=1-\frac\epsilon2.$$
For $n$ sufficiently large,
$y_n(s)$ is an increasing function of $s$.
Since, for $s\ge 0$, $y_n(s)\ge y_n(0)=1-\epsilon/2$ and
$y_n(s)^n\le e^{n(y_n(s)-1)}$,
we  see, again by the comparison principle, that $y_n(s)\ge z_n(s)$ $\forall s\ge 0$, where $z_n$ is the solution of
$$\partial_s z_n(s) = 1-\epsilon -e^{n(z_n(s)-1)},\qquad
z_n(0)=1-\frac\epsilon2.$$
This last equation can be solved explicitly, giving
$$e^{-n(z_n(s)-1)}
=\frac{1-e^{-n(1-\epsilon)s}}{1-\epsilon}
	+e^{-n\big((1-\epsilon) s -\frac\epsilon2\big)}.$$
Hence setting $s=\epsilon$ and letting $n\to\infty$, we obtain that for $\epsilon$ sufficiently small,
$$\lim_{n\to\infty} e^{-n(z_n(\epsilon)-1)}=\frac1{1-\epsilon }.
$$
It follows that as $n\to \infty$,
$$
z_n(\epsilon)=1+\frac{\log(1-\epsilon)}n
+o\left(\frac1n\right).
$$
Therefore
$$
 \lim_{n\to\infty} z_n(\epsilon)^n=1-\epsilon.$$
Since for $\epsilon$ sufficiently small and $n$ sufficiently large we have $u_n(x,t)\geq w_n(0,\epsilon)\geq y_n(\epsilon)\geq z_n(\epsilon)$,
this implies
that for $\epsilon>0$ sufficiently small,
$$\liminf_{n\to\infty} u_n^n(x,t)\ge1-\epsilon,$$
which yields the desired conclusion.
\end{proof}

\begin{lem} \label{lem:uC1}
If the topological boundary $\partial U=\partial S$ between $U$ and $S$ has measure
zero,
then $x\mapsto u(x,t)$ is $C^1$ for every $t>0$, and $\partial_x u$ is continuous on $\R\times (0,\infty)$.
\end{lem}
\begin{proof}
Let 
$$u^*(x,t):=\begin{cases}u(x,t)&\text{if $u(x,t)<1$},\\
	0&\text{if $u(x,t)=1$}.\end{cases}$$
Then we have almost everywhere
\begin{equation}\label{unu*}
u_n(x,t)-u_n^n(x,t)\to u^*(x,t).
\end{equation}
Indeed, this holds
in
$S$ (obviously) and in the interior of $U$ (by
Proposition~\ref{prop:limunn}), and therefore holds almost everywhere by hypothesis.
Hence by \eqref{unrep}, letting $n\to \infty$ and applying dominated convergence, for $t>0$,
\begin{equation} \label{eq:uu*}
u(x,t)=p_t\conv v(x) +\int_0^t \diffd r  \, p_r\conv u^*(x,t-r) .
\end{equation}
Applying Lemma \ref{lemsmooth}, we have that $u(\cdot,t)$ is $C^1$ with
$$\partial_x u(x,t)=p_t'\conv v(x)+\int_0^t  \diffd r\,
p_r'\conv u^*(x,t-r),
$$
and hence, by dominated convergence, $\partial _x u$ is continuous on $\R\times (0,\infty)$, as required.
\end{proof}

\section{Proof of Proposition~\ref{prop mu}} \label{sec:proof1}

In this section, we suppose that $v:\R\to [0,1]$ is a non-increasing function such that $v(x)\to 0$ as $x\to
\infty$ and $v(x)\to 1$ as $x\to-\infty$. 
Let
$\mu_0=\inf\{x\in \R:v(x)<1\}\in \{-\infty\}\cup \R$.
Let $u_n$ denote the solution of~\eqref{mainn}, and define $u$ as in~\eqref{mainu}.
For $t> 0$, let
$$
\mu_t=\inf\Big(\{x\in \R :u(x,t)<1\}\cup \{\infty\}\Big)\in \R\cup\{\infty,-\infty\}.
$$
Note that since $v$ is non-increasing, by the comparison principle we have that $x\mapsto u_n(x,t)$ is non-increasing for each $n$ and each $t\ge 0$, and therefore the same property holds for $u$.
Hence, since $u$ is continuous on $\R\times (0,\infty)$, we have that
for $t>0$
$$
 u(x,t)=1 \, \Leftrightarrow \, x\le \mu_t.
$$
We first prove
that $\mu_t\in \R$ for $t>0$ and bound the increments of $\mu$.
\begin{prop}
\label{propnottoomuchontheleft}
$\mu_t\in\R$ for any $t>0$. Furthermore,
there exists a non-negative
continuous increasing function $\epsilon\mapsto a_\epsilon$ with $a_0=0$
such that
for any $t> 0$ and any $\epsilon\ge0$, 
\begin{equation}
\mu_{t+\epsilon}-\mu_{t} \ge -a_\epsilon.
\label{nottoomuchontheleft}
\end{equation}
If $\mu_0\in\R$, the above also holds at $t=0$.
\end{prop}
\begin{proof}
By~\eqref{eq:FKt}, we have that for $x\in \R$ and $t>0$, 
\begin{equation} \label{eq:uto0}
u(x,t)\leq e^t \E_0\left[v(B_t+x) \right],
\end{equation}
and so, by dominated convergence, $u(x,t)\to 0$ as $x\to \infty$. Hence
$\mu_t<\infty$ $\forall t>0$. 

We now turn to showing that~\eqref{nottoomuchontheleft} holds if $\mu_t\in \R$; we shall then use this to
show that $\mu_t>-\infty$ for $t>0$. 
Take $\underline{v}:\R\to \R$ measurable with $0\le \underline v\le v$,
and, for $x\in \R$ and $t \geq 0$,
let 
\begin{equation}
\label{exlemeq}
\underline u(x,t)=e^t\E_x\big[\underline v(B_t)\big].
\end{equation}
Let $T=\sup\{t\geq 0: \underline u(x,t)<1\ \forall x\in \R\}$; (we call $T$ the time
at which $\underline u$ hits 1).
Then
\begin{equation}\label{exlem}
\underline u(x,t)\le u(x,t)\qquad\forall x\in\R,\, t\le T.
\end{equation}
Indeed, note that $\underline u(x,t)$ is the unique bounded solution to $\partial_t\underline
u = \partial_x^2 \underline u + \underline u$ with initial condition
$\underline v$. By Theorem~\ref{thm u}, for $t<T$, $\underline u(\cdot,t)$ is equal to
the solution of
\eqref{mainn} and \eqref{mainu} with $\underline v$ as initial condition.
Again by Theorem~\ref{thm u}, it follows that $\underline u(\cdot,t)\leq u(\cdot,t)$ for $t<T$. By continuity, we now have~\eqref{exlem}.

Now fix $\epsilon>0$.
Let $\underline
v(x)=\eta \indic{x\in[-a,a]}$ for some fixed $\eta \in (0,1)$ and
$a>0$ to be chosen later. For this choice of $\underline v$, define $\underline u(x,s)$
as in \eqref{exlemeq}. For $\epsilon$ sufficiently small, the pair $(\eta,a)$ can be chosen in such a way
that $\underline u$ hits 1 at time $\epsilon$ (we shall explain below how this is
done); by symmetry, the position where
$\underline u$ hits 1 is $x=0$.

Fix $t\ge0$ such that $\mu_t\in\R$.  Our definition of
$\underline v$ ensures that
$\underline v (x-\mu_t+a) \le u(x,t)$ for all $x\in \R$. Then, by
\eqref{exlem} and the semigroup property in Theorem~\ref{thm u}, $\underline u (x-\mu_t+a,\epsilon)\le
u(x,t+\epsilon)$. In particular, $1=\underline u(0,\epsilon)\le
u(\mu_t-a,t+\epsilon)$ and so $\mu_{t+\epsilon}\ge \mu_t-a$.

We now complete the proof of~\eqref{nottoomuchontheleft} by showing that it is possible, when $\epsilon$
is sufficiently small, to choose
$a=a_\epsilon:=\epsilon^{1/3}$ and to find $\eta\in(0,1)$ such that $\underline u$ hits 1 at time $\epsilon$,
as required.

Introduce 
$$f(s)=\underline u(0,s) = \eta e^s \P_0\left(|B_s|<a_\epsilon \right)=
\eta e^s \int_{-a_\epsilon}^{a_\epsilon} \frac{\diffd y}{\sqrt{4\pi s}}e^{-\frac{y^2}{4s}}.
$$
Note that
$$
\P _0\left(B_\epsilon \geq a_\epsilon \right)=
\int_{a_\epsilon}^\infty \frac{\diffd y}{\sqrt{4\pi \epsilon}}e^{-\frac{y^2}{4\epsilon}}=
\int_{a_\epsilon/\sqrt \epsilon}^\infty \frac{\diffd y}{\sqrt{4\pi}}e^{-y^2/4}
\leq e^{-\epsilon^{-1/3}/4}.
$$
Hence for $\epsilon$ sufficiently small,
$$
e^\epsilon \int_{-a_\epsilon}^{a_\epsilon} \frac{\diffd y}{\sqrt{4\pi \epsilon}}e^{-\frac{y^2}{4\epsilon}}>1,
$$
and we can find $\eta<1$ such that
\begin{equation*}
\underline u (0,\epsilon)=f(\epsilon)=
\eta e^\epsilon \int_{-a_\epsilon}^{a_\epsilon} \frac{\diffd y}{\sqrt{4\pi \epsilon}}e^{-\frac{y^2}{4\epsilon}}=1.
\end{equation*}
It only remains to show that $f(s)<1$ for $s<\epsilon$. To do this, we
simply show that $f'(s)\ge0$ for $s<\epsilon$. We have
\begin{equation}\label{f'(s)}
f'(s)=
\eta e^s \left(
\int_{-a_\epsilon}^{a_\epsilon} \frac{\diffd y}{\sqrt{4\pi s}}e^{-\frac{y^2}{4s}}
-\frac{a_\epsilon}{s}\frac{1}{\sqrt{4\pi s}}e^{-\frac{a_\epsilon^2}{4s}}
\right).
\end{equation}
Clearly, for $\epsilon$ sufficiently small and $s\le\epsilon$, the first term
in the parenthesis of \eqref{f'(s)} is arbitrarily close to 1 while the
second term is arbitrarily close to 0. Hence,  $f'(s)>0$ $\forall s\leq
\epsilon$, which concludes the proof of~\eqref{nottoomuchontheleft}.

Finally, we can now show that in fact $\mu_t>-\infty$ for $t>0$.
Indeed, let $\underline v(x)=\eta\indic{x\in[-a,a]}$
where $a>0$ and $\eta\in(0,1)$ are such that $\underline u$ hits 1 at some time
$s\leq t$. (By the above argument, such a pair $(\eta,a)$ can always be found. By
symmetry, the position where $\underline u$ hits 1 is $x=0$.) Now choose
$x_0$ such that $\underline v(x-x_0)\le v(x)$ $\forall x\in \R$ (this is always
possible as we assumed $v(x)\to 1$ as $x\to-\infty$). Then
by~\eqref{exlem} we have
$\underline u(x-x_0,s)\le u(x,s)$ $\forall x\in \R$ and, in
particular, $1\le u(x_0,s)$, which implies that $\mu_s\ge x_0$.
We now have that $\mu_s \in \R$ for some $s\leq t$ and therefore, by~\eqref{nottoomuchontheleft}, $\mu_t>-\infty$.
\end{proof}

\begin{prop} \label{prop:lag}
The following left-limit  exists for every $t>0$ and satisfies: (l\`ag)
$$\lim_{\epsilon\searrow0}\mu_{t-\epsilon}\le\mu_t.$$
\end{prop}
\begin{proof}
Suppose that the left limit $\lim_{\epsilon\searrow0}\mu_{t-\epsilon}$ does not exist for some $t>0$,
and choose $b$ and $c$ such that
$$\liminf_{\epsilon\searrow0}\mu_{t-\epsilon}<b<c<
\limsup_{\epsilon\searrow0}\mu_{t-\epsilon}.$$
Then for any $\epsilon>0$, there exists $\epsilon'\in (0,\epsilon)$ such that
$\mu_{t-\epsilon'}>c$. 
There also exists
$\epsilon''\in (0,\epsilon')$
such that $\mu_{t-\epsilon''}<b$, so that
$\mu_{t-\epsilon''}-\mu_{t-\epsilon'}<b-c$.
However, by Proposition~\ref{propnottoomuchontheleft},
$\mu_{t-\epsilon''}-\mu_{t-\epsilon'}\geq -a_\epsilon$,
which is a contradiction if $\epsilon$ is sufficiently small that $a_\epsilon<c-b$.
Hence the left limit $\lim_{\epsilon\searrow0}\mu_{t-\epsilon}$ exists. 
By
Proposition~\ref{propnottoomuchontheleft} again,
$\mu_{t-\epsilon}\leq \mu_t +a_\epsilon\to \mu_t$ as $\epsilon\to 0$, and so $\lim_{\epsilon\searrow0}\mu_{t-\epsilon}\le\mu_t.$
\end{proof}
\begin{prop} \label{prop:cad}
The
 map $t\mapsto\mu_t$ is right-continuous  (c\`ad and hence c\`adl\`ag), i.e.~for every $t\geq 0$,
$$\lim_{\epsilon\searrow0} \mu_{t+\epsilon}=\mu_t.$$
\end{prop}
\begin{proof}
Proposition~\ref{propnottoomuchontheleft} already implies that for $t\geq 0$,
$\liminf_{\epsilon\searrow0}\mu_{t+\epsilon}\geq \mu_t$. It now remains to
prove that for any $t\geq 0$,
$\limsup_{\epsilon\searrow0} \mu_{t+\epsilon}\le\mu_t .$
Indeed, fix $t>0$ (we shall consider the case $t=0$ separately). 
For $z>0$, by the definition of $\mu_t$, we have $u(\mu_t+z,t)<1$. 
Then since $u$ is continuous on $\R\times (0,\infty)$,
$u(\mu_t+z,t+\epsilon)<1$ for $\epsilon$ sufficiently small, and so
$\mu_{t+\epsilon}\leq \mu_t+z$.
Hence $\limsup_{\epsilon\searrow0} \mu_{t+\epsilon}\le\mu_t+z$, and the result follows since $z>0$ was arbitrary.

It remains to consider the case $t=0$. First suppose $\mu_0\in\R$ and
take $z>0$.
Since $v$ is non-increasing, 
we have that $v(y)\leq v(\mu_0+z/2)<1$ $\forall y\ge \mu_0+z/2$.
Since $u(\cdot,\epsilon)\to v$ in $L^1_\text{loc}$ as $\epsilon \searrow 0$, and $u(\cdot,\epsilon)$ is non-increasing for $\epsilon>0$, it follows that $u(\mu_0+z,\epsilon)<1$ for $\epsilon$ sufficiently small, and so $\mu_\epsilon <\mu_0 +z$.
Hence for any $z>0$, $\limsup_{\epsilon\searrow0} \mu_{\epsilon}\le\mu_0+z$.
By the same argument, if $\mu_0=-\infty$ then, for any
$z\in\R$,  $u(z,\epsilon)<1$ for $\epsilon$ small enough. Therefore
$\mu_\epsilon<z$ and so for any $z\in \R$, $\limsup_{\epsilon\searrow0}\mu_\epsilon<z$.
\end{proof}

We can finally complete the following important step:
\begin{prop}
The map $t\mapsto\mu_t$ is continuous on $[0,\infty)$.
\end{prop}
\begin{proof}
By Propositions~\ref{prop:cad} and~\ref{prop:lag},
we already have that $t \mapsto \mu_t$ is c\`adl\`ag, and that for $t>0$, 
$\lim_{\epsilon\searrow0}\mu_{t-\epsilon}\le\mu_t.$  Thus the only way in which $\mu$ could fail to be continuous would be if $\lim_{\epsilon\searrow0}\mu_{t-\epsilon} < \mu_t$ for some $t>0$. 
Suppose, for some $t>0$, that
$\lim_{\epsilon\searrow0}\mu_{t-\epsilon}=a  < b=\mu_t$, and take $c\in (a,b).$ Define $f(s)=u(c,s)$ and observe that $ f$ is continuous on $(0,\infty )$.

Since $\lim_{\epsilon\searrow0}\mu_{t-\epsilon}=a $, we have $f(t-\epsilon)<1$ for all $\epsilon >0$ sufficiently small, but since $\mu_t=b$, we have $\lim_{s\to t} f(s)=f(t)=1.$
Fix $t_0 \in (0,t)$ such that $f(s)<1$ $\forall s\in [t_0,t)$, and define $(\tilde u(x,s) , x\in \R, s\ge t_0)$ as the solution of the boundary value problem
\begin{equation}\label{BVP}
\begin{cases}
\partial_t \tilde u =\partial^2_x  \tilde u +\tilde u \qquad &\text{for } x> c \text{ and } s>t_0, \\
\tilde u(c,s)=f(s) \qquad &\text{for } s>t_0, \\
\tilde u(x,t_0)= u(x,t_0) \qquad &\text{for } x\in \R.
\end{cases}
\end{equation}
By Theorem~\ref{thm u}, and since $u(x,s)<1$ for $s\in [t_0,t)$ and $x>c$, we have that $\partial_t u=\partial^2_x u +u$ for $x>c$ and $s\in (t_0,t)$.
Since the solution of the boundary value problem~\eqref{BVP} is unique it
follows that for all $s \in[t_0,t)$ and $ x\ge c$ we have $\tilde u(x,s)
=u(x,s) $. By taking $s\nearrow t$ we also have, by continuity, $\tilde
u(x,t)=u(x,t)$ for $x\ge c$. But since $\mu_t=b$, we must 
have $\tilde u( x,t)=u(x,t)=1$ $\forall x\in [c,b]$. Furthermore, $\lim_{x\to \infty} \tilde u(x,t)=\lim_{x\to \infty} u(x,t) =0$. 
This is impossible because for each $s>t_0$, the solution $\tilde u(\cdot,s)$ of the boundary value problem is analytic (see Theorem 10.4.1 in \cite{Cannon1984}). 
\end{proof}

The proof of Proposition~\ref{prop mu} is now essentially complete. The map
$t \mapsto \mu_t$ is continuous on $[0,\infty )$, whether $\mu_0$ is finite
or $-\infty$. 
Therefore, defining $U$ and $S$ as in~\eqref{eq:UandS}, we see that the
topological boundary between these two domains is simply $\partial
U=\partial S= \{(\mu_t,t):t>0\}$. It has measure zero, and hence by
Lemma~\ref{lem:uC1}, $u(\cdot,t)$ is $C^1$ for every $t>0$ and $\partial_x
u $ is continuous on $\R\times(0,\infty)$.

\section{Proof of uniqueness} \label{sec:uniq}

In this section we prove that the classical solution to~\eqref{FBP} is unique.
We start with the following very simple lemma.
\begin{lem}
If $(u,\mu)$ is a classical solution of~\eqref{FBP}, then for $t>0$,
$$\mu_t=\inf\{x\in \R:u(x,t)<1\}.$$
\end{lem}
\begin{proof}
Suppose, for a contradiction, that $\mu_t<x:=\inf\left(\{x\in \R:u(x,t)<1\}\cup \{\infty\}\right)$ for some $t>0$.
Take $c\in(\mu_t,x)$ and $\epsilon>0$ small enough that, by continuity,
$\mu_{t+s}<c\ \forall s\in[0,\epsilon]$.
Then by Corollary~\ref{lem:FKfbp}, for $y\in (c, x)$ and $\delta \in (0,\epsilon ]$,
$$
u(y,t+\delta)\geq e^\delta \P _y\big(B_s \in [c,x] \,\,
\forall s \leq \delta \big).
$$
This is strictly larger than 1
for $\delta$ sufficiently small, which is a contradiction.
\end{proof}
This lemma implies that if $u_1\le u_2$ then $\mu_1\le\mu_2$, and so the proof of the comparison property of
Theorem~\ref{main thm} will be a consequence of Theorem~\ref{thm u} and the uniqueness of classical solutions of~\eqref{FBP}. Furthermore, it implies that if $(u,\mu)$ and $(\tilde u,\tilde \mu)$ are two classical solutions to~\eqref{FBP}
with the same initial condition~$v$, it is sufficient to show that
 $u=\tilde u$ to obtain that
$\mu=\tilde\mu$.

For $t>0$, let $G_t$ denote the Gaussian semigroup operator, so that for $f\in L^\infty(\R)\cup L^1 (\R)$,
$$
G_t f(x)=p_t\conv f(x)=\int_{-\infty}^\infty \frac{1}{\sqrt{4\pi
t}}e^{-\frac {(x-y)^2}{4t}}f(y)\,\diffd y.
$$
For $m>0$, let $C_m$ denote the cut operator given by
$$C_m f(x)=\min(f(x),m).$$

Suppose that $v:\R\to [0,1]$ is as in Theorem~\ref{main thm}, \textit{i.e.}
$v$ is non-increasing, 
$v(x)\to 0$ as $x \to \infty$ and $v(x)\to1$ as $x\to-\infty$. 
For $n \in \mathbb Z _{\geq 0}$ and $\delta>0$, introduce
\begin{align*}
 u^{n,\delta,-}(x):=\big[e^\delta G_\delta C_{e^{-\delta}}\big]^n v(x)
\quad \text{ and } \quad
u^{n,\delta,+}(x)&:=\big[C_1 e^\delta G_\delta\big]^n v^{\delta,+}(x)
,
\end{align*} 
where we now define $v^{\delta,+}$. Recall that
$\mu_0=\inf \{x\in \R:v(x)<1\}\in\R\cup\{-\infty\}$; 
\begin{align}\text{if $\mu_0\in\R$, let \quad}
v^{\delta,+}(x)&=\begin{cases}1      &\text{if $x<\mu_0+\delta$}\\
				v(x)&\text{if $x\ge\mu_0+\delta$,}
	\end{cases} \label{eq:v+1}\\
\text{and \quad  if $\mu_0=-\infty$, let \quad}
v^{\delta,+}(x)&=\begin{cases}1      &\text{if $v(x)>1-\delta$}\\
				v(x)&\text{if $v(x)\le1-\delta$.}
	\end{cases} \label{eq:v+2}
\end{align}
Our proof of uniqueness relies on the Feynman-Kac representation of
Corollary \ref{lem:FKfbp} and the following two results.
\begin{lem} \label{lem:uniq1}
Suppose $(u,\mu)$ is a classical solution of~\eqref{FBP} with initial condition $v$.
Then for $n \in \mathbb Z _{\geq 0}$, $\delta >0$ and $x\in \R$,
\begin{equation}\label{lem61}
u^{n,\delta,-}(x)\leq u(x,n\delta) \leq u^{n,\delta,+}(x).
\end{equation}
\end{lem}
\begin{lem} \label{lem:uniq2}
For any $\delta>0$, $n \in \mathbb Z _{\geq 0}$, and $A\ge\frac12$,
$$
\int_{-A}^A\Big|u^{n,\delta,+}(x)-u^{n,\delta,-}(x)\Big|\,\diffd x
\leq 
4A(1+e^{\delta n})(e^\delta -1).
$$
\end{lem}
\smallskip

Suppose that $(u,\mu)$ and $(\tilde u,\tilde \mu)$ are classical solutions of~\eqref{FBP} with initial condition $v$.
Then by Lemmas~\ref{lem:uniq1} and~\ref{lem:uniq2}, for $t> 0$, $n \in \mathbb Z_{\geq 0}$ and $A\geq \frac12$,
$$
\int_{-A}^A\Big|u(x,t)-\tilde u(x,t)\Big|\,\diffd x
\leq 
\int_{-A}^A\Big|u^{n,\frac t n,+}(x)-u^{n,\frac t n,-}(x)\Big|\,\diffd x
\leq  4A(1+e^t)(e^{\frac tn }-1).
$$
Since $n\in \mathbb Z_{\geq 0}$ can be taken arbitrarily large, it follows that
$\int_{-A}^A\big|u(x,t)-\tilde u(x,t)\big|\,\diffd x=0$. Letting $A\to \infty$, by continuity of $u(\cdot,t)$ and
$\tilde u(\cdot,t)$ it follows that $u(x,t)=\tilde u(x,t)$ $\forall x\in \R$.
Therefore
 $(u,\mu)$ is the unique classical solution to~\eqref{FBP} with initial condition $v$.

It remains to prove Lemmas~\ref{lem:uniq1}--\ref{lem:uniq2}.
We shall require the following preliminary result for the proof of Lemma~\ref{lem:uniq1}.
\begin{lem} \label{lem:mu+}
Suppose $v^+:\R\to [0,1]$ is non-increasing with $v^+(x)\to 0$ as $x\to
\infty$ and 
$v^+(x)=1$ for some $x\in \R$.
For $t\geq 0$, 
let $u^+(\cdot,t)=e^t G_t v^+(\cdot)$ and
let $$\mu^+_t = \inf\{x\in \R:u^{+}(x,t)<1\}.$$
Then $\mu^+_t\in \R$ $\forall t\geq 0$ and $t\mapsto \mu^+_t$ is continuous.
\end{lem}
This is
a simple result about the heat equation, which can be proved, for instance,
using the
same techniques as in Section~\ref{sec:proof1}.

\begin{proof}[Proof of Lemma~\ref{lem:uniq1}]
We shall show the following result: 
suppose 
that $v:\R\to [0,1]$ is as in Theorem~\ref{main thm}
and
that
$(u,\mu)$ is a classical solution of \eqref{FBP} with
initial condition $v$.
Let $\mu_0=\inf\{x\in \R:v(x)<1\}\in\R\cup\{-\infty\}$.
Suppose $v^-$ and $v^+$
are non-increasing functions with
$$0\le v^-\le v\le v^+\le 1,$$
and that $v^-(x)\to 1$ as $x\to -\infty$, $v^+(x)\to 0$ as $x\to \infty$
 and $\mu^+_0:=\inf\{x\in \R:v^+(x)<1\} >\mu_0$.
Take $\delta >0$. For $t\geq 0$ and $x\in \R$, let
$$
u^+(x,t)=e^t G_t v^+ (x) \quad \text{ and }\quad u^-(x,t)=e^t G_t C_{e^{-\delta}} v^- (x).
$$
Let $\mu^+_t=\inf\{x\in \R:u^+(x,t)<1\}$.
Then we shall prove that
\begin{equation}
u^- (x,\delta)\le u(x,\delta) \le  u^+ (x,\delta) \,\, \forall x\in \R \quad \text{ and } \quad \mu_\delta^+>\mu_\delta.
\label{comp}
\end{equation}
Since $u(x,\delta)\in [0,1]$, it follows from~\eqref{comp} that
\begin{equation*}
 0\leq e^\delta G_\delta C_{e^{-\delta}}\,v^-(x)\le u(x,\delta) \le  C_1 e^\delta
G_\delta \,v^+(x)\leq 1,
\end{equation*}
and~\eqref{lem61} follows by the definition of $v^{\delta,+}$ and by induction on $n$.

We now prove~\eqref{comp}.
 Let
$\tau=\inf\{s\geq 0:B_s\leq \mu_{\delta-s}\}\wedge \delta $.
Then by Corollary~\ref{lem:FKfbp},
for $x\in \R$,
\begin{align*}
u(x,\delta)
&=\E_x \left[e^\delta v(B_\delta) \indic{\tau= \delta}
+e^\tau \indic{\tau < \delta}\right]\\
&\geq \E_x \left[e^\delta \min(v(B_\delta),e^{-\delta}) \indic{\tau= \delta}
+e^\delta \min(v(B_\delta),e^{-\delta}) \indic{\tau < \delta}\right]\\
&= \E_x \left[e^\delta \min(v(B_\delta),e^{-\delta}) \right]\\
&= u^{-}(x,\delta).
\end{align*}

Now let $t_0:= \inf\{t\geq 0:\mu^+_t \leq \mu_t\}$. By continuity of
$\mu_t$ and $\mu^+_t$ (from Lemma~\ref{lem:mu+}), we have $t_0>0$. We will show below that
$t_0=\infty$.
Take $t< t_0$ and, again, let
$\tau=\inf\{s\geq 0:B_s\leq \mu_{t-s}\} \wedge t$. By Proposition~\ref{lem:FKw} we have
\begin{equation*} \label{eq:uniq(dagger)}
u^{+}(x,t)=\E_x \left[
e^t v^+(B_t)\indic{\tau=t}+e^\tau u^{+}(B_\tau, t-\tau)\indic{\tau <t}\right].
\end{equation*}
Then, again by Corollary~\ref{lem:FKfbp},
\begin{align*}
u(x,t)
&=\E_x \left[e^t v(B_t)\indic{\tau=t}+e^\tau \indic{\tau<t}\right]\\
&\leq \E_x \left[e^t v^{+}(B_t)\indic{\tau=t}+e^\tau u^{+}(B_\tau,t-\tau)\indic{\tau<t}\right]\\
&=u^{+}(x,t),
\end{align*}
where the second line follows since $v\leq v^+$ and since, on $\{\tau<t\}$,
we have $B_\tau=\mu_{t-\tau}<\mu^+_{t-\tau}$ 
and so $u^{+}(B_\tau,t-\tau)\geq 1$.
By continuity, the inequality also holds for $t=t_0$ and so
\begin{equation} \label{eq:uniqu(fivestar)}
u^{+}(x,t)\geq u(x,t)\,\,\forall x\in \R, t\leq t_0.
\end{equation}

Suppose, for a contradiction, that $t_0<\infty$.
Then, by continuity, $\mu_{t_0}^+=\mu_{t_0}$.
Hence $u(\mu_{t_0},t_0)=1=u^{+}(\mu_{t_0},t_0)$
and $\partial_x u(\mu_{t_0},t_0)=0$,
and so by~\eqref{eq:uniqu(fivestar)}, $\partial_x u^{+}(\mu_{t_0},t_0)= 0$.

Note that $u^+$ is smooth on $\R\times (0,\infty)$ and, by the same argument as in Lemma~\ref{lem:uchi}, for $t>0$, $\partial _x u^+(\cdot,t/2)$ is bounded.
Therefore for $x\in \R$,
$$
\partial _x u^+(x,t)
=e^{t/2} p_{t/2} \ast \partial _x u^+(x,t/2)<0
$$
since $u^+(\cdot,t/2)$ is a non-increasing non-constant function.

We now have a contradiction. Therefore $t_0=\infty$ and we have $\mu_\delta^+>\mu_\delta$ and $u^+(x,\delta)\geq u(x,\delta)$ $\forall x\in \R$ by~\eqref{eq:uniqu(fivestar)}.
This completes the proof of~\eqref{comp}.
\end{proof}

\begin{proof}[Proof of Lemma~\ref{lem:uniq2}]
Some of the ideas in this proof are from Section~4.3 of~\cite{deMasi2017b}.

In this proof, we use both the supremum norm $\|\ \|_\infty$ and the $L^1$
norm $\|\ \|_1$. When a property holds for both norms, we simply write it
with $\|\ \|$.

Note the following basic properties of our operators: for $f,g \in
L^\infty(\R)\cup L^1(\R)$,  $m>0$ and $t> 0$, we have for either norm that
\begin{equation} \label{eq:uniqa}
\|C_m f-C_m g\| \leq \|f-g\|,\quad
\|G_t f\| \leq \|f\|.\quad 
\end{equation}
For the supremum norm, we also have that
\begin{equation} \label{eq:uniqb}
\|C_m f-f\|_\infty \leq \max\big( \|f\|_\infty -m,0\big).
\end{equation}
For $w:\R\to [0,\infty)$, $\delta>0$ and $x\in \R$,
\begin{equation} \label{eq:uniq1}
C_1 e^\delta w(x)=\min(e^\delta w(x),1)=e^\delta\min( w(x),e^{-\delta})
=e^\delta C_{e^{-\delta}}w(x).
\end{equation}
Using \eqref{eq:uniq1}, we can rewrite $u^{n,\delta,-}$ as
\begin{equation} \label{eq:un-}
u^{n,\delta,-}
=
  [e^\delta G_\delta C_{e^{-\delta}}]^n\, v
= [G_\delta C_1 e^\delta ]^n\, v
= G_\delta [C_1 e^\delta G_\delta ]^{n-1}C_1 e^\delta \, v.
\end{equation}
We can also write $u^{n,\delta,+}$ as
\begin{equation} \label{eq:un+}
u^{n,\delta,+}
= [C_1 e^\delta G_\delta ]^n\, v^{\delta,+}
= C_1 e^\delta G_\delta   [C_1 e^\delta G_\delta ]^{n-1}\, v^{\delta,+}.
\end{equation}
Now let
$$
f:=e^\delta G_\delta   [C_1 e^\delta G_\delta ]^{n-1}\, v^{\delta,+}
 -e^\delta G_\delta   [C_1 e^\delta G_\delta ]^{n-1}\, v,
$$
and let $g:=u^{n,\delta,+}-u^{n,\delta,-}-f$. By the triangle inequality, we have 
\begin{align*}
\big\|g\big\|_\infty
&\le
\big\|u^{n,\delta,+}-e^\delta G_\delta [C_1 e^\delta G_\delta ]^{n-1}\,
v^{\delta,+}\big\|_\infty 
+
\big\|e^\delta G_\delta [C_1 e^\delta G_\delta ]^{n-1}\,
v-u^{n,\delta,-}\big\|_\infty .
\end{align*}
By our expression for $u^{n,\delta,+}$ in~\eqref{eq:un+} and the properties of $G_\delta$ and $C_1$ in~\eqref{eq:uniqa} and~\eqref{eq:uniqb} respectively,
the first term on the right hand side is bounded above by $e^\delta-1$. A second application of the
triangle inequality then yields
\begin{align*}
\big\|g\big\|_\infty
&\le e^\delta-1+(e^\delta-1) \big\| G_\delta [C_1 e^\delta G_\delta ]^{n-1}\,
v\big\|_\infty+\big\| G_\delta [C_1 e^\delta G_\delta ]^{n-1}\,
v-u^{n,\delta,-}\big\|_\infty .
\end{align*}
Clearly $\big\| G_\delta [C_1 e^\delta G_\delta ]^{n-1}\,
v\big\|_\infty\le 1$ by \eqref{eq:uniqa}. Replacing $u^{n,\delta,-}$ by its expression in~\eqref{eq:un-} gives
\begin{align*}
\big\|g\big\|_\infty
&\le 2(e^\delta-1)+
\big\|G_\delta [C_1 e^\delta G_\delta ]^{n-1}v-
G_\delta [C_1 e^\delta G_\delta ]^{n-1}
C_1e^\delta v\big\|_\infty
\\&\le 2(e^\delta-1)+e^{\delta(n-1)}
\big\|v-C_1e^\delta v\big\|_\infty ,
\end{align*}
where \eqref{eq:uniqa} was used repeatedly in the second inequality.
But $\big\|v-C_1e^\delta v\|_\infty
    =\big\|C_1v-C_1e^\delta v\|_\infty\le e^\delta -1$
   by~\eqref{eq:uniqa},
and so 
\begin{align} \label{eq:g}
\big\|g\big\|_\infty
&\le 2(e^\delta-1)+ e^{\delta(n-1)}(e^\delta-1)
\le (2+e^{\delta n})(e^\delta -1).
\end{align}
By~\eqref{eq:uniqa} applied repeatedly, for either norm we have
$$\|f\|\le e^{\delta n}\|v^{\delta,+}-v\|.$$
We now need to consider two cases.
\begin{itemize}
\item
If $\mu_0=-\infty$, then, by our definition of $v^{\delta,+}$ in~\eqref{eq:v+2}, we have 
$\|v^{\delta,+}-v\|_\infty=\delta$ and so $\|f\|_\infty\le e^{\delta
n}\delta$.
\item If $\mu_0\in\R$, then by our definition of $v^{\delta,+}$ in~\eqref{eq:v+1}, $\|v^{\delta,+}-v\|_1\le\delta$   and so
$\|f\|_1\le e^{\delta n}\delta$. 
\end{itemize}
In either case, if $A\ge\frac12$ then
$$\int_{-A}^A |f(x)|\,\diffd x\le 2A e^{\delta n}\delta\le 2Ae^{\delta
n}(e^\delta-1).$$
By~\eqref{eq:g}, we also have
$$\int_{-A}^A |g(x)|\,\diffd x\le 2A (2+e^{\delta n})
(e^\delta-1).$$
By a final application of the triangle inequality 
to $u^{n,\delta,+}-u^{n,\delta,-}=f+g$,
the result follows.
\end{proof}


\section{Proof of the Feynman-Kac results from Section \ref{sec:FK}}\label{sec:FK proofs}

Before proving Proposition~\ref{lem:FKw}, we need the following result:

\begin{lem} \label{lem:min0}
Let $f:[0,1]\to \R \cup \{-\infty\}$ be continuous with $f(0)<0$ and $f(1)<0$. 
Let $(\xi_t)_{t\in [0,1]}$ denote a Brownian bridge (with diffusivity $\sqrt 2$) from 0 to 0 in time $1$.
Then
\[
\P\left( \min_{s\le 1} (\xi_s-f(s))=0\right) =0.
\]
\end{lem}

\begin{proof}
By a union bound, we have that
\begin{align*}
&\P\left( \min_{s\le 1} (\xi_s-f(s))=0\right)\\
& \quad \leq \P\left( \min_{s\le 1/2} (\xi_s-f(s))=0\right)
+ \P\left( \min_{s\le 1/2} (\xi_{1-s}-f(1-s))=0\right).
\end{align*}
Given any fixed continuous function $b:[0,\infty)\to \R$ with $b(0)=0$, there is at most one value of $z\in \R$ such that 
\[
\min_{s\le 1/2} \left\{b(s) +2sz -f(s)\right\} =0.
\]
Thus,
recalling the definition of $p_t$ in~\eqref{ptdef},
\begin{align*}
&\P \left( \min_{s\le 1/2} (\xi_s-f(s))=0\right)\\
 & \quad  = \E\left[ \P\left[ \left. \min_{s\le 1/2} (\xi_s-f(s))=0 \right|\xi_{1/2}\right]\right]
\\
& \quad = \int_{-\infty}^\infty \diffd z\, p_{1/4}(z) \P\left(   \min_{s\le 1/2} \left\{\tfrac{1}{\sqrt 2} \xi_{2s} +2sz -f(s)\right\} =0 \right)
\\
& \quad= \E \left[  \int_{-\infty}^\infty \diffd z \,p_{1/4}(z) \indic{ \min_{s\le 1/2} \{\frac{1}{\sqrt 2} \xi_{2s} +2sz -f(s)\} =0 }   \right]
\\
& \quad =0,
\end{align*}
where the second equality holds since $\xi_s \sim N(0,2s(1-s))$,
the third equality follows by Fubini's theorem and
the last equality follows because for each realisation of $(\xi_s)_{s\in [0,1]}$ there is at most one value of $z$ for which the integrand is non-zero.
By the same argument,
$$ \P\left( \min_{s\le 1/2} (\xi_{1-s}-f(1-s))=0\right)=0,$$
and the result follows.
\end{proof}

\begin{proof}[Proof of Proposition \ref{lem:FKw}]
Fix $(x,t)\in A$. We begin by proving the result under condition~1. For $\sigma\in [0,\tau]$, let 
$$M_\sigma=w(B_\sigma,t-\sigma)e^{I_\sigma}+\int_0^\sigma\diffd r\,
S(B_r,t-r) e^{I_r},\qquad\text{where }I_\sigma=\int_0^\sigma
K(B_s,t-s)\,\diffd s .$$
Since $w$ is $C^{2,1}$ on $A$, for $\sigma\leq \tau$, we apply 
It\^o's formula (with no leading $\frac12$ in front of the $\partial_x^2$
term because $(B_s)_{s\geq 0}$ has diffusivity $\sqrt 2$):
$$
\begin{aligned}
\diffd M_\sigma &=
\hphantom{+} \partial_x  w(B_\sigma,t-\sigma)e^{I_\sigma}\diffd B_\sigma
+\partial_x^2w(B_\sigma,t-\sigma)e^{I_\sigma}\diffd   \sigma
\\&\hphantom{=}-\partial_tw(B_\sigma,t-\sigma)e^{I_\sigma}\diffd\sigma
+          w(B_\sigma,t-\sigma)e^{I_\sigma}K(B_\sigma,t-\sigma)\diffd \sigma
+ S(B_\sigma,t-\sigma) e^{I_\sigma}\diffd\sigma
\\
&=\hphantom{+}\partial_x  w(B_\sigma,t-\sigma)e^{I_\sigma}\diffd B_\sigma+
[{-\partial_tw+\partial_x^2w+Kw+S}](B_\sigma,t-\sigma)e^{I_\sigma}\diffd   \sigma
\\&=\hphantom{+}\partial_x  w(B_\sigma,t-\sigma)e^{I_\sigma}\diffd
B_\sigma,
\end{aligned}
$$
where we used \eqref{equ w} in the last line, since $(B_\sigma,t-\sigma)\in A$ for $\sigma\leq \tau$. We see that 
$(M_\sigma)_{\sigma \leq \tau}$ is a local martingale.
Therefore, since $(M_\sigma)_{\sigma \leq \tau}$ is bounded, we have that
$$
w(x,t)=\E_x[M_0]=\E_x[M_\tau],
$$
which yields the result~\eqref{eq:FKgen} under condition~1.

We now turn to condition~2, 
with $A=\{(x,t): t\in(0,T), x>\mu_t\}$ and
$\tau=\inf\big\{s\ge0:B_s\le \mu_{t-s}\big\}\wedge t$. The stopping time
$\tau$ is the first time that $(B_\tau,t-\tau)\in\partial A$.
For $\epsilon>0$ and $\delta>0$, introduce the stopping times
$$\tau_{\epsilon,\delta}=\inf\big\{s\ge0:B_s\le \mu_{t-s}+\delta\big\}\wedge
(t-\epsilon),\quad
\tau_{\epsilon}=\inf\big\{s\ge0:B_s\le \mu_{t-s}\big\}\wedge
(t-\epsilon)=\tau\wedge (t-\epsilon).
$$
By~\eqref{eq:FKgen} under condition~1 with stopping time $\tau_{\epsilon,\delta}$ we have that for $x>\mu_t$,
$$w(x,t)=\E_x
\left[w(B_{\tau_{\epsilon,\delta}},t-\tau_{\epsilon,\delta})e^{\int_0^{\tau_{\epsilon,\delta}} K(B_s,t-s)\,\diffd s} 
  + \int_0^{\tau_{\epsilon,\delta}}\diffd r\, S(B_r,t-r)e^{\int_0^{r   } K(B_s,t-s)\,\diffd
s}\right].$$
We now take the limit $\delta\to0$. Since $\mu$ is continuous,
$\tau_{\epsilon,\delta}\to\tau_\epsilon$ as $\delta\to0$, and since $w$,
$K$ and $S$ are bounded, and $w$ is continuous on $\bar A \cap (\R\times (0,\infty))$, we 
obtain, by continuity and dominated convergence,
\begin{equation} \label{eq:w_epsilon}
w(x,t)=\E_x
\left[w(B_{\tau_{\epsilon}},t-\tau_{\epsilon})e^{\int_0^{\tau_{\epsilon}} K(B_s,t-s)\,\diffd s} 
  + \int_0^{\tau_{\epsilon}}\diffd r\, S(B_r,t-r)e^{\int_0^{r   } K(B_s,t-s)\,\diffd
s}\right].
\end{equation}

We now take the limit $\epsilon\to0$ to prove \eqref{eq:FKgen} under
condition~2. Note that $\tau_\epsilon\nearrow\tau$ as $\epsilon\searrow0$;
as $S$ and $K$ are
bounded, by dominated convergence
$$
\E_x \left[
   \int_0^{\tau_\epsilon}\diffd r\, S(B_r,t-r)e^{\int_0^{r   } K(B_s,t-s)\,\diffd s}\right]
\to
\E_x \left[
   \int_0^{\tau}\diffd r\, S(B_r,t-r)e^{\int_0^{r   } K(B_s,t-s)\,\diffd
s}\right]\qquad\text{as $\epsilon\searrow0$}.
$$
We now turn to the first term on the right hand side of~\eqref{eq:w_epsilon}. As above, for $r\geq 0$, let
$I_r=\int_0^{r}
K(B_s,t-s)\,\diffd s$.
Write
\begin{equation*}
\E_x \left[ w(B_{\tau_\epsilon},t-\tau_\epsilon)e^{I_{\tau_\epsilon}}\right]
=
\E_x \left[ w(B_{\tau_\epsilon},t-\tau_\epsilon)e^{I_{\tau_\epsilon}
}\indic{\tau<t}\right]
+
\E_x \left[ w(B_{t-\epsilon},\epsilon)e^{I_{t-\epsilon}}
 \indic{\tau=t}\right]
\end{equation*}
(we used that $\tau_\epsilon=t-\epsilon$ when $\tau=t$).
Since $w$ and $K$ are bounded, and $w$ is continuous on $\bar A \cap (\R \times (0,\infty))$, by
continuity and dominated convergence the first term on the right hand side converges to 
$\E_x \left[ w(B_\tau,t-\tau)e^{I_\tau}\indic{\tau<t}\right] $ as $\epsilon \searrow 0$. 
For the second term, write
\begin{equation}
\label{3.3}
\begin{aligned}
\E_x\Big[
w(B_{t-\epsilon},\epsilon)e^{I_{t-\epsilon}}
 \indic{\tau=t}
\Big]
=&\quad
\E_x\Big[ w(B_{t-\epsilon},\epsilon)e^{I_{t-\epsilon}}\Big(\indic{\tau=t}
-\indic{\tau\ge t-\epsilon}\Big)\Big]
\\&+
\E_x\Big[ w(B_{t-\epsilon},\epsilon)e^{I_{t-\epsilon}}\indic{\tau\ge
t-\epsilon}\Big]
-
\E_x\Big[ w(B_{t},\epsilon)e^{I_t}\indic{\tau=t}\Big]
\\&+
\E_x\Big[ w(B_{t},\epsilon)e^{I_t}\indic{\tau=t}\Big].
\end{aligned}
\end{equation}
(In this equation, we set $w(B_t,\epsilon)$ to an arbitrary bounded value when $B_t<\mu_\epsilon$.)
It is clear by dominated convergence that  the first line on the right hand side
of \eqref{3.3}
goes to 0 as $\epsilon\searrow0$.

Let us now show that the second line of \eqref{3.3} goes to 0 as $\epsilon\searrow0$.
Define
$\phi_r(y;x,t)$ as
\begin{equation} \label{eq:phitdef}
\phi_r(y;x,t)=\E_x\Big[\delta(B_r-y)e^{I_r}\indic{\tau\ge r} \Big]
=p_r(x-y)\E_x\Big[e^{I_r}\indic{\tau\ge r} \Big| B_r=y\Big].
\end{equation}
(This is the probability, weighted by $e^{I_r}$, that the path of length $r$
started from $(x,t)$ arrives at $(y,t-r)$ without touching the left 
boundary.) By integrating over the value of $B_r$, 
we have that for $r\leq t$,
\begin{equation*}\label{Ephi}
\E_x\Big[ w(B_{r},\epsilon)e^{I_r}\indic{\tau\ge r}\Big]
=\int_{-\infty}^{\infty}\diffd y\,w(y,\epsilon)
\phi_{r         }(y;x,t).
\end{equation*}
Thus,  the second line of~\eqref{3.3} can be written as
\begin{equation}
\E_x\Big[
w(B_{t-\epsilon},\epsilon)e^{I_{t-\epsilon}}\indic{\tau\ge t-\epsilon}\Big]
-
\E_x\Big[ w(B_{t},\epsilon)e^{I_t}\indic{\tau=t}\Big]
=\int_{-\infty}^{\infty}\diffd y\,w(y,\epsilon)
\Big(
 \phi_{t-\epsilon}(y;x,t)
-\phi_{t         }(y;x,t)
\Big).
\label{secondline}
    \end{equation}
To show that the second line of~\eqref{3.3} goes to zero as $\epsilon\searrow0$, it
is then sufficient to show that $
\big| \phi_{t-\epsilon}(y;x,t)
-\phi_{t         }(y;x,t)
\big|$ is bounded by an integrable function
and goes to 0 as $\epsilon\searrow0$.

Let $(\xi_t)_{t\in [0,1]}$ denote a Brownian bridge (with diffusivity
$\sqrt2$) from 0 to 0 in time $1$ and introduce,
 for $x$, $y$ fixed,
\begin{equation*}
F(\xi,s,r)=\sqrt r \xi +x +s\frac{y-x}r,  
\end{equation*}
so that $(F(\xi_{s/r},s,r))_{s\leq r}$
is a Brownian bridge from $x$ to $y$ in time $r$.
Then by~\eqref{eq:phitdef},
\begin{align*}
\phi_r(y;x,t)
&=p_r(x-y)\E_x\Big[e^{I_r}\indic{\tau\ge r} \Big| B_r=y\Big] \\
&=p_r(x-y)\E_x\Big[e^{\int_0^rK(B_s,t-s)\diffd s}\indic{\mu_{t-s}< B_s\
\forall s < r} \Big| B_r=y\Big]
\\
&=p_r(x-y)\E\Big[e^{\int_0^rK\big(F(\xi_{s/r},s,r),t-s\big)\diffd s}
\indic{\mu_{t-s}< F(\xi_{s/r},s,r)\
\forall s < r} \Big]
.
\end{align*}
Set $r=t-\epsilon$ and take the $\epsilon\searrow0$ limit. 
Now by the continuity of $\xi$,
$$
\begin{aligned}
\indic{\mu_{t-s}< F(\xi_{s/t},s,t) \,\,\forall s<t}
	\le& \liminf_{\epsilon\searrow0} \indic{\mu_{t-s}<
F(\xi_{s/(t-\epsilon)},s,t-\epsilon)\,\,\forall s<t-\epsilon}\\
	\le&
\limsup_{\epsilon\searrow0}\indic{\mu_{t-s}<
F(\xi_{s/(t-\epsilon)},s,t-\epsilon)\,\,\forall s<t-\epsilon}
\le \indic{\mu_{t-s}\le F(\xi_{s/t},s,t)\,\,\forall s<t}.
\end{aligned}
$$
Since $x>\mu_t$, for $y>\mu_0$ we can apply
Lemma~\ref{lem:min0}, which yields that the probability that the lower and upper bounds
above are different is zero.
We can conclude, by dominated convergence,
that
\begin{align*}
\lim_{\epsilon\searrow0}{\phi_{t-\epsilon}(y;x,t)}
={\phi_{t}(y;x,t)}.
\end{align*}
For $y<\mu_0$, we have that $y<\mu_\epsilon$ for $\epsilon$ sufficiently small, and so for $\epsilon$ sufficiently small,
$$
\phi_{t-\epsilon}(y;x,t)=0
={\phi_{t}(y;x,t)}.
$$ 
Since $t>0$,
and $\phi_{t-\epsilon}(y;x,t) \leq p_{t-\epsilon}(x-y) e^{t \|K\|_\infty}$ by~\eqref{eq:phitdef},
 it is easy to see that $\phi_{t-\epsilon}(\cdot;x,t)$ can be
uniformly bounded for $\epsilon<t/2$
by a function with Gaussian tails. Therefore, by dominated convergence we see
that \eqref{secondline} (which is the second line of \eqref{3.3}) goes to
0 as $\epsilon\searrow0$.

It only remains to consider the third line of \eqref{3.3}. Using~\eqref{eq:phitdef}, we can write
\begin{equation}
\E_x\Big[ w(B_{t},\epsilon)e^{I_t}\indic{\tau=t}\Big]
=\int_{-\infty}^{\infty}\diffd y\,w(y,\epsilon)
\phi_{t         }(y;x,t).
\end{equation}
Since $w$ is bounded, 
$\phi_{t}(y;x,t) \leq p_{t}(x-y) e^{t \|K\|_\infty}$
and
$w(\cdot,\epsilon)\to w(\cdot,0)$ in $L^1_\text{loc}$ as
$\epsilon\searrow 0$, we have that
\begin{equation*}
\int_{-\infty}^\infty \diffd y\, w(y,\epsilon)\phi_t(y;x,t) \to
\int_{-\infty}^\infty \diffd y\, w(y,0)\phi_t(y;x,t)
=\E_x\Big[ w(B_{t},0)e^{I_t}\indic{\tau=t}\Big]
\quad \text{ as }\epsilon \searrow 0 .
\end{equation*}
This completes the proof.\end{proof}

\medskip

\noindent {\bf Acknowledgements:} JB and EB thank E. Presutti and A. De Masi for their hospitality at the GSSI, Italy where part of this work was conducted. \\
EB was partially supported by ANR grant ANR-16-CE93-0003 (ANR MALIN).\\
JB was partially supported by ANR grants ANR-14-CE25-0014 (ANR GRAAL)
and ANR-14-CE25-0013 (ANR NONLOCAL).

\bibliographystyle{abbrv}
\bibliography{FBP}

\end{document}